\documentclass[12pt]{amsart}
\usepackage{amsfonts,amssymb,amscd,amsmath,enumerate,color, xypic}
\def\NZQ{\mathbb}               % the font for N,Z,Q,R,C

\def\ZZ{{\NZQ Z}}
\def\RR{{\NZQ R}}

\def\frk{\mathfrak}               % font for "Fraktur"

\def\Phi{{\frk N}}
\def\eb{{\bold e}}

\def\opn#1#2{\def#1{\operatorname{#2}}} % to make operators
\opn\chara{char} 
\opn\length{\ell} 
\opn\pd{pd} 
\opn\rk{rk}
\opn\projdim{proj\,dim} 
\opn\injdim{inj\,dim} 
\opn\rank{rank}
\opn\depth{depth} 
\opn\grade{grade} 
\opn\height{height}
\opn\embdim{emb\,dim} 
\opn\codim{codim}

\opn\Tr{Tr} 
\opn\bigrank{big\,rank}
\opn\superheight{superheight}
\opn\lcm{lcm}
\opn\trdeg{tr\,deg}%\emph{
\opn\reg{reg} 
\opn\lreg{lreg} 
\opn\ini{in} 
\opn\lpd{lpd}
\opn\size{size}
\opn\mult{mult}
\opn\dist{dist}
\opn\cone{cone}
\opn\lex{lex}
\opn\rev{rev}
\opn\div{div} \opn\Div{Div} \opn\cl{cl} \opn\Cl{Cl}
\opn\Spec{Spec} \opn\Supp{Supp} \opn\supp{supp} \opn\Sing{Sing}
\opn\Ass{Ass} \opn\Min{Min}
\opn\Ann{Ann} \opn\Rad{Rad} \opn\Soc{Soc}
\opn\Syz{Syz} \opn\Im{Im} \opn\Ker{Ker} \opn\Coker{Coker}
\opn\Am{Am} \opn\Hom{Hom} \opn\Tor{Tor} \opn\Ext{Ext}
\opn\End{End} \opn\Aut{Aut} \opn\id{id} \opn\ini{in}

\opn\nat{nat}
\opn\pff{pf}%   \pf exists already
\opn\Pf{Pf} \opn\GL{GL} \opn\SL{SL} \opn\mod{mod} \opn\ord{ord}
\opn\Gin{Gin}
\opn\Hilb{Hilb}\opn\adeg{adeg}\opn\std{std}\opn\ip{infpt}
\opn\Pol{Pol}
\opn\sat{sat}
\opn\Var{Var}
\opn\Gen{Gen}

\opn\aff{aff} \opn\con{conv} \opn\relint{relint} \opn\st{st}
\opn\lk{lk} \opn\cn{cn} \opn\core{core} \opn\vol{vol}
\opn\link{link} \opn\star{star}
\opn\gr{gr}

\def\Jc{{\mathcal J}}
\def\Gc{{\mathcal G}}

\def\Oc{{\mathcal O}}

\def\pot#1#2{#1[\kern-0.28ex[#2]\kern-0.28ex]}

\opn\dirlim{\underrightarrow{\lim}}
\opn\inivlim{\underleftarrow{\lim}}
\let\to=\rightarrow

\def\Implies{\ifmmode\Longrightarrow \else
        \unskip${}\Longrightarrow{}$\ignorespaces\fi}
\def\implies{\ifmmode\Rightarrow \else
        \unskip${}\Rightarrow{}$\ignorespaces\fi}
\def\iff{\ifmmode\Longleftrightarrow \else
        \unskip${}\Longleftrightarrow{}$\ignorespaces\fi}

\let\:=\colon
\newtheorem{Theorem}{Theorem}[section]
\newtheorem{Lemma}[Theorem]{Lemma}
\newtheorem{Corollary}[Theorem]{Corollary}

\newtheorem{Example}[Theorem]{Example}

\newtheorem{Conjecture}[Theorem]{Conjecture}

\let\epsilon\varepsilon
\let\phi=\varphi
\let\kappa=\varkappa
\def\qed{\ifhmode\textqed\fi
      \ifmmode\ifinner\quad\qedsymbol\else\dispqed\fi\fi}
\def\textqed{\unskip\nobreak\penalty50
       \hskip2em\hbox{}\nobreak\hfil\qedsymbol
       \parfillskip=0pt \finalhyphendemerits=0}
\def\dispqed{\rlap{\qquad\qedsymbol}}

\opn\dis{dis}
\opn\height{height}
\opn\dist{dist}
\def\pnt{{\raise0.5mm\hbox{\large\bf.}}}

\opn\Lex{Lex}
\opn\conv{conv}

\begin{document}
%%%%%%%%%%%%%%%%%%%%%%%%%%%%%%%%%%%%%%%%
%% title
%%%%%%%%%%%%%%%%%%%%%%%%%%%%%%%%%%%%%%%%
\title{Quadratic Gr\"obner bases of twinned order polytopes}
\author[T.~Hibi]{Takayuki Hibi}
\address[Takayuki Hibi]{Department of Pure and Applied Mathematics,
	Graduate School of Information Science and Technology,
	Osaka University,
	Toyonaka, Osaka 560-0043, Japan}
\email{hibi@math.sci.osaka-u.ac.jp}
\author[K.~Matsuda]{Kazunori Matsuda}
\address[Kazunori Matsuda]{Department of Pure and Applied Mathematics,
	Graduate School of Information Science and Technology,
	Osaka University,
	Toyonaka, Osaka 560-0043, Japan}
\email{kaz-matsuda@math.sci.osaka-u.ac.jp}
%\subjclass[2010]{00000}
%\keywords{Gr\"obner basis}
\begin{abstract}
Let $P$ and $Q$ be finite partially ordered sets on $[d] = \{1, \ldots, d\}$,
and $\Oc(P) \subset \RR^{d}$ and $\Oc(Q) \subset \RR^{d}$ their order polytopes.
The twinned order polytope of $P$ and $Q$ 
is the convex polytope $\Delta(P,-Q) \subset \RR^{d}$
which is the convex hull of $\Oc(P) \cup (- \Oc(Q))$.
It follows that the origin of $\RR^{d}$ belongs to the interior of $\Delta(P,-Q)$
if and only if $P$ and $Q$ possess a common linear extension.  
It will be proved that, when the origin of $\RR^{d}$ belongs to the interior of 
$\Delta(P,-Q)$, the toric ideal of $\Delta(P,-Q)$ possesses a quadratic Gr\"obner basis
with respect to a reverse lexicographic order for which the variable corresponding to 
the origin is smallest.  Thus in particular 
if $P$ and $Q$ possess a common linear extension, then the twinned order polytope
$\Delta(P,-Q)$ is a normal Gorenstein Fano polytope. 
\end{abstract}
\maketitle
\section*{Introduction}
In \cite{HMOS}, from a viewpoint of Gr\"obner bases, 
the centrally symmetric configuration (\cite{CSC}) of the order polytope
(\cite{Stanley}) of a finite partially ordered set is studied.
In the present paper,  
a far-reaching generalization of \cite{HMOS} will be discussed. 
 
Let $P = \{ p_{1}, \ldots, p_{d} \}$ and $Q = \{ q_{1}, \ldots, q_{d} \}$ 
be finite partially ordered sets (posets, for short)
with $|P| = |Q| = d$.
A subset $I$ of $P$ is called a {\em poset ideal} of $P$ if
$p_{i} \in I$ and $p_{j} \in P$ together with $p_{j} \leq p_{i}$ guarantee
$p_{j} \in I$.  Thus in particular the empty set $\emptyset$ as well as $P$ itself
is a poset ideal of $P$. 
Write $\Jc(P)$ for the set of poset ideals of $P$ and
$\Jc(Q)$ for that of $Q$.
A {\em linear extension} of $P$ is a permutation $\sigma = i_{1}i_{2}\cdots i_{d}$ 
of $[d] = \{ 1, \ldots, d \}$ for which $i_{a} < i_{b}$ if $p_{i_{a}} < p_{i_{b}}$.

Let $\eb_{1}, \ldots, \eb_{d}$ stand for the canonical unit coordinate
vectors of $\RR^{d}$.  Then, for each subset $I \subset P$ and 
for each subset $J$ of $Q$, we define
$\rho(I) = \sum_{p_{i}\in I} \eb_{i}$ and
$\rho(J) = \sum_{q_{j}\in J} \eb_{j}$.
In particular $\rho(\emptyset)$ is the origin ${\bf 0}$ of $\RR^{d}$.
Define $\Omega(P, - Q) \subset \ZZ^{d}$ as 
\[
\Omega(P, - Q) = 
\{ \, \rho(I) \, : \, \emptyset \neq I \in \Jc(P) \, \} \cup
\{ \, - \rho(J) \, : \, \emptyset \neq J \in \Jc(Q) \, \} 
\cup \{ {\bf 0} \} 
\]
and write $\Delta(P,-Q) \subset \RR^{d}$ for the convex polytope 
which is the convex hull of $\Omega(P, - Q)$.
We call $\Delta(P,-Q)$ the {\em twinned order polytope} of $P$ and $Q$.
In other words, the twinned order polytope 
$\Delta(P,-Q)$ of $P$ and $Q$
is the convex polytope
which is the convex hull of $\Oc(P) \cup ( - \Oc(Q))$, where 
$\Oc(P) \subset \RR^{d}$ is the order polytope of $P$
and $- \Oc(Q) = \{ - \beta \, ; \, \beta \in \Oc(Q) \}$.
One has $\dim \Delta(P,-Q) = d$.  
Since $\rho(P) = \rho(Q) = \eb_{1} + \cdots + \eb_{d}$,
it follows that the origin ${\bf 0}$ of $\RR^{d}$
cannot be a vertex of $\Delta(P,-Q)$.  In fact,
the set of vertices of $\Delta(P,-Q)$ are $\Omega(P, - Q) \setminus \{\bf 0\}$.

This paper is organized as follows.
In Section $1$, a basic fact that the origin of $\RR^{d}$ belongs to the interior 
of $\Delta(P,-Q)$ if and only if $P$ and $Q$ possess a common linear extension
(Lemma \ref{Sapporo}).
We then show, in Section $2$, that,
when the origin of $\RR^{d}$ belongs to the interior of $\Delta(P,-Q)$, 
the toric ideal of $\Delta(P,-Q)$ possesses a quadratic Gr\"obner basis
with respect to a reverse lexicographic order for which 
the variable corresponding to the origin is smallest
(Theorem \ref{Boston}).
Thus in particular 
if $P$ and $Q$ possess a common linear extension, then the twinned order polytope
$\Delta(P,-Q)$ is a normal Gorenstein Fano polytope
(Corollary \ref{Berkeley}). 
Finally, we conclude this paper with a collection of examples in Section $3$.
We refer the reader to \cite{dojoEN} for fundamental materials on Gr\"obner bases
and toric ideals.

\section{Linear extensions}
Let $P$ and $Q$ be finite posets with $|P| = |Q| = d$.
In general, the origin ${\bf 0}$ of $\RR^{d}$ may not belong to the interior of 
the twinned order polytope $\Delta(P,-Q)$ of $P$ and $Q$.
It is then natural to ask when 
the origin of $\RR^{d}$ belongs to the interior of $\Delta(P,-Q)$.

\begin{Lemma}
\label{Sapporo}
Let $P = \{ p_{1}, \ldots, p_{d} \}$ and $Q = \{ q_{1}, \ldots, q_{d} \}$ 
be finite posets.
Then the following conditions are equivalent{\rm :}
\begin{enumerate}
\item[{\rm (i)}]
The origin of $\RR^{d}$ belongs to the interior of $\Delta(P,-Q)${\rm ;}
\item[{\rm (ii)}] 
$P$ and $Q$ possess a common linear extension.
% There exists a permutation $\sigma$ of $[d]$ such that
% $\sigma$ is a linear extension of each of $P$ and $Q$. 
\end{enumerate} 
% Thus, in particular, if $P$ is a refinement of $Q$, then
% the origin of $\RR^{d}$ belongs to the interior of $\Delta(P,-Q)$.
\end{Lemma} 

\begin{proof}
((i) $\Rightarrow$ (ii))
Suppose that the origin ${\bf 0}$ of $\RR^{d}$ belongs to the interior of $\Delta(P,-Q)$.
Since $\Omega(P, - Q) \setminus \{\bf 0\}$ is the set of vertices of
$\Delta(P,-Q)$, the existence of an equality 
\begin{eqnarray}
\label{interior}
{\bf 0} = \sum_{\emptyset \neq I \in \Jc(P)} a_{I} \cdot \rho(I)
\ + \sum_{\emptyset \neq J \in \Jc(Q)} b_{J} \cdot (- \rho(J)),
\end{eqnarray}
where each of $a_{I}$ and $b_{J}$ is a positive real numbers, is guaranteed.
Let
\[
\sum_{\emptyset \neq I \in \Jc(P)} a_{I} \cdot \rho(I) 
= \sum_{i=1}^{d} a^{*}_{i} \eb_{i},
\, \, \, \, \, 
\sum_{\emptyset \neq J \in \Jc(Q)} b_{J} \cdot \rho(J) 
= \sum_{i=1}^{d} b^{*}_{i} \eb_{i},
\]
where each of $a^{*}_{i}$ and $b^{*}_{i}$ is a positive rational number.
Since each $I$ is a poset ideal of $P$ and each $J$ is a poset ideal of $Q$,
it follows that $a^{*}_{i} > a^{*}_{j}$ if $p_{i} < p_{j}$.
Let $\sigma = i_{1} i_{2} \cdots i_{d}$ be a permutation of $[d]$
for which $i_{a} < i_{b}$ if $a^{*}_{i} > a^{*}_{j}$.
Then $\sigma$ is a linear extension of $P$.  
Furthermore, by using (\ref{interior}), one has  
$a^{*}_{i} = b^{*}_{i}$ for $1 \leq i \leq d$.  It then turn out that
$\sigma$ is also a linear extension of $Q$, as required.  

((ii) $\Rightarrow$ (i))
Let $\sigma = i_{1}i_{2}\cdots i_{d}$ be a linear extension of 
each of $P$ and $Q$.
Then $I_{k} = \{p_{i_{1}}, \ldots, p_{i_{k}}\}$ is a poset ideal of $P$
and $J_{k} = \{q_{i_{1}}, \ldots, q_{i_{k}}\}$ is a poset ideal of $Q$
for $1 \leq k \leq d$.  Hence
\begin{eqnarray}
\label{point}
\pm\eb_{i_{1}}, \, \pm(\eb_{i_{1}} + \eb_{i_{2}}), \ldots, \, 
\pm(\eb_{i_{1}} + \cdots + \eb_{i_{d}})
\end{eqnarray}
belong to $\Omega(P, - Q)$.  Let $\Gamma \subset \RR^{d}$ denote the convex polytope
which is the convex hull of (\ref{point}).
Since $\dim \Gamma = d$ and since the origin
of $\RR^{d}$ belongs to the interior of $\Gamma$,
it follows that the origin of $\RR^{d}$ belongs to
the interior of $\Delta(P, - Q)$, as desired.  
\end{proof}

\section{Quadratic Gr\"obner bases}
Let, as before, $P = \{ p_{1}, \ldots, p_{d} \}$ and $Q = \{ q_{1}, \ldots, q_{d} \}$ 
be finite partially ordered sets.
Let $K[{\bf t}, {\bf t}^{-1}, s] 
= K[t_{1}, \ldots, t_{d}, t_{1}^{-1}, \ldots, t_{d}^{-1}, s]$
denote the Laurent polynomial ring in $2d + 1$ variables over a field $K$. 
If $\alpha = (\alpha_{1}, \ldots, \alpha_{d}) \in \ZZ^{d}$, then
${\bf t}^{\alpha}s$ is the Laurent monomial
$t_{1}^{\alpha_{1}} \cdots t_{d}^{\alpha_{d}}s$. 
In particular ${\bf t}^{\bf 0}s = s$.
The {\em toric ring} of $\Omega(P, - Q)$ is the subring 
$K[\Omega(P, - Q)]$ of $K[{\bf t}, {\bf t}^{-1}, s]$ which is generated
by those Laurent monomials ${\bf t}^{\alpha}s$ 
with $\alpha \in \Omega(P, - Q)$.
Let 
\[
K[{\bf x}, {\bf y}, z] = K[\{x_{I}\}_{\emptyset \neq I \in \Jc(P)} \cup 
\{y_{J}\}_{\emptyset \neq J \in \Jc(Q)} \cup \{ z \}]
\]
denote the polynomial ring in $|\Omega(P, - Q)|$ variables over $K$
and define the surjective ring homomorphism 
$\pi : K[{\bf x}, {\bf y}, z] \to K[\Omega(P, - Q)]$
by setting 
\begin{itemize}
\item
$\pi(x_{I}) = {\bf t}^{\rho(I)}s$, 
where $\emptyset \neq I \in \Jc(P)$;
\item
$\pi(y_{J}) = {\bf t}^{- \rho(J)}s$,
where $\emptyset \neq J \in \Jc(Q)$;
\item
$\pi(z) = s$.
\end{itemize}
The {\em toric ideal} $I_{\Omega(P, - Q)}$ of $\Omega(P, - Q)$ is 
the kernel of $\pi$. 

Let $<$ denote a reverse lexicographic order on $K[{\bf x}, {\bf y}, z]$
satisfying
\begin{itemize}
\item
$z < x_{I}$ and $z < y_{J}$;
\item
$x_{I'} < x_{I}$ if $I' \subset I$;
\item
$y_{J'} < y_{J}$ if $J' \subset J$,
\end{itemize}
and $\Gc$ the set of the following binomials:
\begin{enumerate}
\item[(i)]
$x_{I}x_{I'} - x_{I\cap I'}x_{I \cup I'}$;
\item[(ii)]
$y_{J}y_{J'} - y_{J\cap J'}y_{J \cup J'}$;
\item[(iii)]
% $x_{I}y_{J} - 
% x_{I \setminus \{p_{i}\}_{i \in T}}
% y_{J \setminus \{q_{j}\}_{j \in T}}$,
$x_{I}y_{J} - x_{I \setminus \{p_{i}\}}y_{J \setminus \{q_{i}\}}$,
\end{enumerate}
where 
\begin{itemize}
\item
$x_{\emptyset} = y_{\emptyset} = z$;
\item
$I$ and $I'$ belong to $\Jc(P)$ and 
$J$ and $J'$ belong to $\Jc(Q)$;
% \item
% $T \subset [\,d\,] = \{1, \ldots, d\}$;
% \item
% $\{p_{i}\}_{i \in T} \subset I$
% and $I \setminus \{p_{i}\}_{i \in T} \in \Jc(P)$;
% \item
% $\{q_{j}\}_{j \in T} \subset J$
% and
% $J \setminus \{q_{j}\}_{j \in T} \in \Jc(Q)$.
\item $p_{i}$ is a maximal element of $I$ and $q_{i}$ is a maximal element of $J$.
\end{itemize}

\begin{Theorem}
\label{Boston}
Work with the same situation as above.
% Suppose that $P$ is a refinement of $Q$.
Suppose that $P$ and $Q$ possess a common linear extension. 
Then $\Gc$ is a Gr\"obner basis of
$I_{\Omega(P, - Q)}$ with respect to $<$.
\end{Theorem}

\begin{proof}
It is clear that $\Gc \subset I_{\Omega(P, - Q)}$.
In general, if $f = u - v$ is a binomial, then $u$ is called the {\em first}
monomial of $f$ and $v$ is called the {\em second} monomial of $f$. 
The initial monomial of each of the binomials (i) -- (iii) 
with respect to $<$ is its first monomial. 
Let ${\rm in}_{<}(\Gc)$ denote the set of initial monomials of binomials 
belonging to $\Gc$.  It follows from \cite[(0.1)]{OHrootsystem} that,
% H.~Ohsugi and T.~Hibi, 
% Quadratic initial ideals of root systems,
% Proc. Amer. Math. Soc. {\bf 130} (2002), 1913--1922.
in order to show that $\Gc$ is a Gr\"obner basis of
$I_{\Omega(P, - Q)}$ with respect to $<$, what we must prove is
the following:
($\clubsuit$) If $u$ and $v$ are monomials belonging to 
$K[{\bf x}, {\bf y}, z]$ with $u \neq v$ such that 
$u \not\in \langle {\rm in}_{<}(\Gc) \rangle$ 
and $v \not\in \langle {\rm in}_{<}(\Gc) \rangle$,
then $\pi(u) \neq \pi(v)$.

Let $u$ and $v$ be monomials belonging to $K[{\bf x}, {\bf y}, z]$
with $u \neq v$.  Write
\[
u = z^{\alpha} x_{I_{1}}^{\xi_{1}} \cdots x_{I_{a}}^{\xi_{a}}
y_{J_{1}}^{\nu_{1}} \cdots y_{J_{b}}^{\nu_{b}},
\, \, \, \, \, \, \, \, \, \, 
v = z^{\alpha'} x_{I'_{1}}^{\xi'_{1}} \cdots x_{I'_{a'}}^{\xi'_{a'}}
y_{J'_{1}}^{\nu'_{1}} \cdots y_{J'_{b'}}^{\nu'_{b'}},
\]
where
\begin{itemize}
\item
$\alpha \geq 0$, $\alpha' \geq 0$;
\item
$I_{1}, \ldots, I_{a}, I'_{1}, \ldots, I'_{a'} 
\in \Jc(P) \setminus \{ \emptyset \}$;
\item
$J_{1}, \ldots, J_{b}, J'_{1}, \ldots, J'_{b'} 
\in \Jc(Q) \setminus \{ \emptyset \}$;
\item
$\xi_{1}, \ldots, \xi_{a}, 
\nu_{1}, \ldots, \nu_{b},
\xi'_{1}, \ldots, \xi'_{a'}, 
\nu'_{1}, \ldots, \nu'_{b'} > 0$,
\end{itemize}
and where $u$ and $v$ are relatively prime with
$u \not\in \langle {\rm in}_{<}(\Gc) \rangle$ 
and $v \not\in \langle {\rm in}_{<}(\Gc) \rangle$.
Especially either $\alpha = 0$ or $\alpha' = 0$.
Let, say, $\alpha' = 0$.  Thus
\[
u = z^{\alpha} x_{I_{1}}^{\xi_{1}} \cdots x_{I_{a}}^{\xi_{a}}
y_{J_{1}}^{\nu_{1}} \cdots y_{J_{b}}^{\nu_{b}},
\, \, \, \, \, \, \, \, \, \, 
v = x_{I'_{1}}^{\xi'_{1}} \cdots x_{I'_{a'}}^{\xi'_{a'}}
y_{J'_{1}}^{\nu'_{1}} \cdots y_{J'_{b'}}^{\nu'_{b'}}.
\]
% where
% \[
% \alpha + \xi_{1} + \cdots + \xi_{a} + \nu_{1} + \cdots + \nu_{b}
% = \xi'_{1} + \cdots + \xi'_{a'} + \nu'_{1} + \cdots + \nu'_{b'}.
% \]
By using (i) and (ii), it follows that
\begin{itemize}
\item
$I_{1} \subset I_{2} \subset \cdots \subset I_{a}, \, 
I_{1} \neq I_{2} \neq \cdots \neq I_{a}$;
\item
$J_{1} \subset J_{2} \subset \cdots \subset J_{b}, \, 
J_{1} \neq J_{2} \neq \cdots \neq J_{b}$;
\item
$I'_{1} \subset I'_{2} \subset \cdots \subset I'_{a'}, \, 
I'_{1} \neq I'_{2} \neq \cdots \neq I'_{a'}$;
\item
$J'_{1} \subset J'_{2} \subset \cdots \subset J'_{b'}, \, 
J'_{1} \neq J'_{2} \neq \cdots \neq J'_{b'}$.
\end{itemize}
Furthermore, by virtue of \cite{Hibi1987}, it suffices to discuss 
$u$ and $v$ with $(a, a') \neq (0, 0)$ and $(b, b') \neq (0,0)$.

Let $A_{i}$ denote the power of $t_{i}$ appearing in 
$\pi(x_{I_{1}}^{\xi_{1}} \cdots x_{I_{a}}^{\xi_{a}})$
and $A'_{i}$ the power of $t_{i}$ appearing in 
$\pi(x_{I'_{1}}^{\xi'_{1}} \cdots x_{I'_{a'}}^{\xi'_{a'}})$.
Similarly let $B_{i}$ denote the power of $t_{i}^{-1}$ appearing in 
$\pi(y_{J_{1}}^{\nu_{1}} \cdots y_{J_{b}}^{\nu_{b}})$
and $B'_{i}$ the power of $t_{i}^{-1}$ appearing in 
$\pi(y_{J'_{1}}^{\nu'_{1}} \cdots y_{J'_{b'}}^{\nu'_{b'}})$.

Since $P$ and $Q$ possess a common linear extension,
after relabeling the elements of $P$ and $Q$, we assume that
if $p_{r} < p_{s}$ in $P$, then $r < s$, and
if $q_{r'} < q_{s'}$ in $Q$, then $r' < s'$.

Let $1 \leq j_{*} \leq d$ denote the biggest integer 
for which one has $A_{j_{*}} \neq A'_{j_{*}}$.  Since $I_{a} \neq I'_{a'}$,
the existence of $j_{*}$ is guaranteed.
Let $j_{*} = d$ and, say,
$A_{d} > A'_{d}$.  Then $p_{d} \in I_{a}$.  
Since $p_{d}$ is a maximal element of $P$ and
$q_{d}$ is that of $Q$,
by using (iii), it follows that $q_{d}$ cannot belong to $J_{b}$.
Hence $\pi(u) \neq \pi(v)$, as desired.

Let $j_{*} < d$ and $A_{j_{*}} > A'_{j_{*}}$.
Let $1 \leq e \leq a$ denote the integer with $p_{j_{*}} \in I_{e}$ 
and $p_{j_{*}} \not\in I_{e - 1}$.  
We claim that $p_{j_{*}}$ is a maximal element of $I_{e}$.
To see why this is true, let $p_{j_{*}} < p_{h}$ in $I_{e}$.
Then $j_{*} < h$. 
Since both $p_{j_{*}}$ and $p_{h}$ belong to
each of $I_{e}, I_{e+1}, \ldots, I_{a}$, it follows that
$A_{j_{*}} = A_{h}$.
Now, since $p_{j_{*}} < p_{h}$,
one has $A'_{j_{*}} \geq A'_{h}$.
Hence $A_{h} = A_{j_{*}} > A'_{j_{*}} \geq A'_{h}$.
However, the definition of $j_{*}$ says that
$A_{h} = A'_{h}$, a contradiction.  Hence 
$p_{j_{*}}$ is a maximal element of $I_{e}$.

Now, suppose that $\pi(u) = \pi(v)$.
Then $B_{d} = B'_{d}, \ldots, B_{j_{*}+1} = B'_{j_{*}+1}$ and
$B_{j_{*}} > B'_{j_{*}}$.  Then the above argument guarantees  
the existence of $J_{e'}$ for which $q_{j_{*}}$ is a maximal element
of $J_{e'}$.  The fact that $p_{j_{*}}$ is a maximal element of $I_{e}$
and $q_{j_{*}}$ is that of $J_{e'}$
contradicts (iii).  As a result, one has $\pi(u) \neq \pi(v)$, as desired. 
\, \, \, \, \, \, \, \, \, \, \, 
\end{proof}

Theorem \ref{Boston} is a far-reaching generalization of \cite[Theorem 2.2]{HMOS}.
We refer the reader to \cite{HMOS} and \cite{harmony} for basic materials on
normal Gorenstein Fano polytopes.
As in \cite[Corollary 2.3]{HMOS} and \cite[Corollary 1.3]{harmony}, 
it follows that

\begin{Corollary}
\label{Berkeley}
If $P$ and $Q$ possess a common linear extension, then the twinned order polytope
$\Delta(P,-Q)$ is a normal Gorenstein Fano polytope. 
\end{Corollary}

\section{Examples}
We conclude this paper with a collection of examples.  It is natural to ask, 
if, in general, the toric ideal of $I_{\Omega(P, - Q)}$ possesses
a quadratic Gr\"{o}bner basis with respect to a reverse lexicographic order
as in Theorem \ref{Boston}.

\begin{Example}
{\em
In general, a toric ideal $I_{\Omega(P, - Q)}$ may not possess a quadratic 
Gr\"{o}bner basis with respect to a reverse lexicographic order $<$ introduced as above.
Let $P = \{p_{1}, \ldots, p_{5}\}$ and $Q = \{q_{1}, \ldots, q_{5}\}$ 
be the following finite posets: 

%be the finite posets with
%\[
%p_1 < p_3, \, \, p_2 < p_3, \, \, p_2 < p_4, \, \, p_3 < p_5, \, \, p_4 < p_5; 
%\]
%\[
%q_4 < q_3, \, \, q_4 < q_5, \, \, q_3 < q_2, \, \, q_2 < q_1. 
%\]
%%%%%%%%

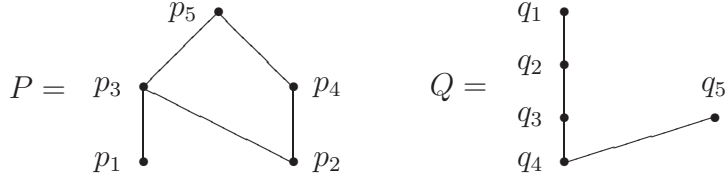
\vspace{5mm}

\begin{xy}
	\ar@{} (0,  0) ; (30,  0)  *{P = } , 
	\ar@{} (0,  0) ; (44,  -10)  *++!R{p_1}*\dir<4pt>{*} = "A", 
	\ar@{} "A"; (64,  -10) *++!L{p_2}*\dir<4pt>{*} = "B", 
	\ar@{-} "A"; (44,  0) *++!R{p_3}*\dir<4pt>{*} = "C", 
	\ar@{-} "B"; (64,  0) *++!L{p_4}*\dir<4pt>{*} = "D", 
	\ar@{-} "B" ; "C", 
	\ar@{-} "C" ; (54, 10) *++!R{p_5}*\dir<4pt>{*} = "I", 
	\ar@{-} "D" ; "I", 
	\ar@{} (0,  0) ; (86,  0)  *{Q = } , 
	\ar@{} (0,  0) ; (100,  -10) *++!R{q_4}*\dir<4pt>{*} = "E", 
	\ar@{-} "E" ; (100,  -4) *++!R{q_3}*\dir<4pt>{*} = "F", 
	\ar@{-} "F" ; (100,  3) *++!R{q_2}*\dir<4pt>{*} = "G",
	\ar@{-} "G" ; (100,  10) *++!R{q_1}*\dir<4pt>{*} = "H",  
	\ar@{-} "E" ; (120,  -4) *++!D{q_5}*\dir<4pt>{*}, 
\end{xy}

\vspace{5mm}

\noindent
Since $p_1 < p_3$ and $q_3 < q_1$, it follows that no linear extension of $P$ 
is a linear extension of $Q$.  Then a routine computation guarantees that,
for any reverse lexicographic order as in Theorem \ref{Boston}, the binomial 
\begin{eqnarray}
\label{hello}
x_{\{2\}} x_{\{1, 2, 3, 4\}} y_{\{1, 2, 3, 4, 5\}} - x_{\{2, 4\}} y_{\{4, 5\}}z
\end{eqnarray}
belongs to the reduced Gr\"obner basis of $I_{\Omega(P, - Q)}$ with respect to $<$.
However, the toric ideal $I_{\Omega(P, - Q)}$ is generated by quadratic binomials.
The $S$-polynomial of the binomials 
\[
x_{\{2, 4\}} x_{\{1, 2 , 3\}} - x_{\{2\}} x_{\{1, 2, 3, 4\}}, \, \, \, 
x_{\{1, 2, 3\}} y_{\{1, 2, 3, 4, 5\}} - y_{\{4, 5\}} z
\]
belonging to a system of generators of $I_{\Omega(P, - Q)}$ coincided with
the binomial (\ref{hello}).
}
\end{Example}

\begin{Conjecture}
Let % $P = \{ p_{1}, \ldots, p_{d} \}$ and $Q = \{ q_{1}, \ldots, q_{d} \}$
$P$ and $Q$ be arbitrary finite posets with $|P| = |Q| = d$.
Then the toric ideal $I_{\Omega(P, - Q)}$
is generated by quadratic binomials.
\end{Conjecture}

Let $\delta(\Delta(P,-Q))$ denote the $\delta$-vector (\cite[p.~79]{HibiRedBook}) 
of $\Delta(P,-Q)$.
It then follows that, if $P$ and $Q$ possess a common linear extension, then
$\delta(\Delta(P,-Q))$ is symmetric and unimodal.

\begin{Example}
{\em
Let $P = \{ p_{1}, \ldots, p_{d} \}$ be a chain and $Q = \{ q_{1}, \ldots, q_{d} \}$
an antichain.  Then the $\delta$-vector of $\Delta(P,-Q)$ is
\begin{eqnarray*}
d = 2 &:& (1, 3, 1), \\
d = 3 &:& (1, 7, 7, 1), \\
d = 4 &:& (1, 15, 33, 15, 1), \\
d = 5 &:& (1, 31, 131, 131, 31, 1), \\
d = 6 &:& (1, 63, 473, 883, 473, 63, 1).
\end{eqnarray*}
It seems likely that $\delta(\Delta(P,-Q))$ coincides with
the Pascal-like triangle \cite[pp.~11--12]{Barry} with $r = 1$.   
}
\end{Example}

\end{document}